\documentclass{louloupart}
\usepackage{pinlabel}
\usepackage{graphicx}
\usepackage[all]{xy}
\title[Parabolic subgroups inside parabolic subgroups of Artin groups]{Parabolic subgroups inside parabolic subgroups of Artin groups}
\author[M\,A Blufstein]{Mart\'in\,A Blufstein}
\givenname{Mart\'in\,A}
\surname{Blufstein}
\address{Mart\'in\,A Blufstein, Departamento de Matem\'atica - IMAS, FCEyN, Universidad de Buenos Aires, Buenos Aires, Argentina}
\email{mblufstein@dm.uba.ar}
\author[L Paris]{Luis Paris}
\givenname{Luis}
\surname{Paris}
\address{Luis Paris, IMB, UMR 5584, CNRS, Universit\'e Bourgogne Franche-Comt\'e, 21000 Dijon, France}
\email{lparis@u-bourgogne.fr}


\newtheorem{thm}{Theorem}[section]
\newtheorem{lem}[thm]{Lemma}
\newtheorem{prop}[thm]{Proposition}

\theoremstyle{definition}

\newtheorem*{acknow}{Acknowledgments}

\numberwithin{equation}{section}

\makeatletter
\renewcommand{\thefigure}{\ifnum \c@section>\z@ \thesection.\fi
 \@arabic\c@figure}
\@addtoreset{figure}{section}
\makeatother

\begin{document}

\def\Prod{{\rm Prod}} \def\N{\mathbb N} \def\CA{{\rm CA}}
\def\Sal{{\rm Sal}} \def\SSal{\overline{\Sal}} \def\id{{\rm id}}


\begin{abstract}
We prove that a parabolic subgroup $P$ contained in another parabolic subgroup $P'$ of an Artin group $A$ is a parabolic subgroup of $P'$.
This answers a question of Godelle which is not obvious despite appearances.
In order to achieve our result we construct a set-retraction $A \to P$ of the inclusion map from a parabolic subgroup $P$ into $A$.
This retraction was implicitly constructed in a previous paper by Charney and the second author.
\end{abstract}

\maketitle


\section{Introduction}\label{sec1}

If $a,b$ are two letters and $m$ is an integer greater or equal to $2$, then we denote by $\Prod(a,b,m)$ the alternating word $aba \cdots$ of length $m$.
We take a finite simplicial graph $\Gamma$ and we denote by $V(\Gamma)$ its set of vertices and by $E(\Gamma)$ its set of edges. 
We endow $E(\Gamma)$ with a labeling $m : E(\Gamma) \to \N_{\ge 2}$ and we take an abstract set $\Sigma=\{ \sigma_x \mid x \in V(\Gamma)\}$ in one-to-one correspondence with $V(\Gamma)$.
Then the \emph{Artin group} $A=A[\Gamma]$ of $\Gamma$ is defined by the presentation 
\[
A = \langle \Sigma \mid \Prod(\sigma_x,\sigma_y,m(e)) = \Prod (\sigma_y, \sigma_x, m(e)) \text{ for } e=\{x,y\} \in E(\Gamma) \rangle\,.
\]

Let $X$ be a subset of $V(\Gamma)$.
We denote by $\Gamma_X$ the full subgraph of $\Gamma$ spanned by $X$ and we endow $E(\Gamma_X)$ with the labeling induced by that of $E(\Gamma)$.
We set $\Sigma_X = \{ \sigma_x \mid x \in X\}$ and we denote by $A_X$ the subgroup of $A$ generated by $\Sigma_X$.
We know by van der Lek \cite{Lek1} that $A_X$ is naturally isomorphic to $A[\Gamma_X]$, hence we will not differentiate $A_X$ from $A[\Gamma_X]$. 
The subgroup $A_X$ is called a \emph{standard parabolic subgroup} of $A$ and a subgroup conjugate to $A_X$ is called a \emph{parabolic subgroup} of $A$.

An important question in the study of Artin groups is to determine whether the intersection of two parabolic subgroups is a parabolic subgroup.
This question is solved for right angled Artin groups by Duncan--Kazachkov--Remeslennikov \cite{DuKaRe1}, for Artin groups of spherical type by Cumplido--Gebhardt--Gonz\'alez-Meneses--Wiest \cite{CGGW1}, for Artin groups of large type by Cumplido--Martin--Vaskou \cite{CuMaVa1}, and for some two dimensional Artin groups by the first author \cite{Blufs1}.
It is also partially solved when the Artin group is of FC type by Morris-Wright \cite{Morri1} (see also M\"oller--Paris--Varghese \cite{MoPaVa1}).

In this paper we prove that a parabolic subgroup $P$ of $A$ contained in another parabolic subgroup $P'$ is a parabolic subgroup of $P'$.
We do this for all Artin groups.
Results proved for all Artin groups are quite uncommon in the literature.
In general, they involve only certain families of Artin groups, so our paper is in some sense a rarity.
This result is a preliminary to the above question, and it was a question posed by Godelle \cite[Conjecture 2]{Godel1}. 
Additionally, it is a central step towards solving the conjugacy stability problem for Artin groups (see \cite{Cump1}).
The question seems obvious but is not. 
It is also related to the study of normalizers and centralizers of parabolic subgroups.
In more precise terms we prove the following. 

\begin{thm}\label{thm1_1}
Let $\Gamma$ be a finite simplicial graph, let $m: E(\Gamma) \to \N_{\ge 2}$ be a labeling, and let $A = A[\Gamma]$ be the Artin group of $\Gamma$. 
Let $X,Y \subset V(\Gamma)$ and $\alpha \in A$ such that $\alpha A_Y \alpha^{-1} \subset A_X$.
Then there exist $Y' \subset X$ and $\gamma \in A_X$ such that $\alpha A_Y \alpha^{-1} = \gamma A_{Y'} \gamma^{-1}$.
\end{thm}

Theorem \ref{thm1_1} was proved in Rolfsen \cite{Rolfs1} and in Fenn--Rolfsen--Zhu \cite{FeRoZh1} for braid groups, in Paris \cite{Paris1} and in Godelle \cite{Godel2} for Artin groups of spherical type, in Godelle \cite{Godel3} for Artin groups of FC type, in Godelle \cite{Godel1} for two dimensional Artin groups and in Haettel \cite{Haett1} for some Euclidean type Artin groups. 
Our proof is independent from these works and it is valid for all Artin groups.

Let $X \subset V(\Gamma)$.
In order to achieve our goal we construct a set-retraction $\pi_X : A \to A_X$ to the inclusion map $A_X \hookrightarrow A$ (see Proposition \ref{prop2_3}).
This map is defined directly on the words that represent the elements of $A$, but it is not a homomorphism, although its restriction to the so-called colored subgroup is a homomorphism. 
The construction of this map is interesting by itself and it can be considered as an important result of the paper.
However, we underline that this construction is implicit in the proof of Theorem 1.2 of Charney--Paris \cite{ChaPar1} and our contribution consists in making it explicit.

\begin{acknow}
The first author is supported by CONICET. The second author is supported by the French project ``AlMaRe'' (ANR-19-CE40-0001-01) of the ANR.
\end{acknow}


\section{Proofs}\label{sec2}

We keep the notations from Section \ref{sec1}.
So, $\Gamma$ is a finite simplicial graph whose set of edges is endowed with a labeling $m: E(\Gamma) \to \N_{\ge 2}$ and $A=A[\Gamma]$ is the Artin group of $\Gamma$.

Let $S=\{s_x \mid x \in V(\Gamma)\}$ be an abstract set in one-to-one correspondence with $V(\Gamma)$.
Then the \emph{Coxeter group} $W=W[\Gamma]$ of $\Gamma$ is defined by the presentation
\begin{gather*}
W = \langle S \mid \Prod(s_x,s_y,m(e)) = \Prod (s_y, s_x, m(e)) \text{ for } e=\{x,y\} \in E(\Gamma)\,,\\
s_x^2=1 \text{ for } x \in V(\Gamma) \rangle\,.
\end{gather*}

Let $X$ be a subset of $V(\Gamma)$.
We set $S_X=\{s_x \mid x \in X\}$ and we denote by $W_X$ the subgroup of $W$ generated by $S_X$.
We know by Bourbaki \cite{Bourb1} that $W_X$ is naturally isomorphic to $W[\Gamma_X]$, hence, as for Artin groups, we will not differentiate $W_X$ from $W[\Gamma_X]$. 
The subgroup $W_X$ is called a \emph{standard parabolic subgroup} of $W$ and a subgroup conjugate to $W_X$ is called a \emph{parabolic subgroup} of $W$.

We denote by $\theta : A \to W$ the natural epimorphism which sends $\sigma_x$ to $s_x$ for all $x\in V(\Gamma)$.
The kernel of $\theta$ is denoted by $\CA = \CA[\Gamma]$ and it is called the \emph{colored Artin group} of $\Gamma$.
The epimorphism $\theta$ has a natural set-section $\iota : W \to A$ defined as follows. 
For $w\in W$ the word length of $w$ with respect to $S$ is denoted by $\ell_S(w)$, and an expression $w=s_{x_1} s_{x_2} \cdots s_{x_p}$ is called \emph{reduced} if $p= \ell_S(w)$.
Let $w \in W$.
We choose a reduced expression $w = s_{x_1} s_{x_2} \cdots s_{x_p}$ and we set $\iota (w) = \sigma_{x_1} \sigma_{x_2} \cdots \sigma_{x_p}$.
By Tits \cite{Tits1} this definition does not depend on the choice of the reduced expression.
Notice that $\iota$ is not a homomorphism, but, if $u,v \in W$ are such that $\ell_S(uv) = \ell_S(u) + \ell_S(v)$, then $\iota(uv) = \iota(u)\,\iota(v)$.
We clearly have $\theta \circ \iota = \id$. 

For $X \subset V(\Gamma)$ we set $\CA_X = \CA \cap A_X$.
Since the inclusion map from $\Gamma_X$ to $\Gamma$ induces isomorphisms $W[\Gamma_X] \to W_X$ and $A[\Gamma_X] \to A_X$, the isomomorphism $A[\Gamma_X] \to A_X$ restricts to an isomorphism $\CA[\Gamma_X] \to \CA_X$.
So, as for $W_X$ and $A_X$, we will not differentiate $\CA_X$ from $\CA[\Gamma_X]$.

The following lemma arises from the exercises of Chapter 4 of Bourbaki \cite{Bourb1} (see also Davis \cite[Section 4.3]{Davis1}) and it is widely used in the study of Coxeter groups.

\begin{lem}[Bourbaki \cite{Bourb1}]\label{lem2_1}
Let $X,Y \subset V(\Gamma)$ and let $w \in W$.
\begin{itemize}
\item[(1)]
There exists a unique element of minimal length in the double-coset $W_X\,w\,W_Y$.
\item[(2)]
Let $w_0$ be the element of minimal length in $W_X\,w\,W_Y$.
For each $v \in W_X\,w\,W_Y$ there exist $u_1 \in W_X$ and $u_2 \in W_Y$ such that $v=u_1 w_0 u_2$ and $\ell_S(v) = \ell_S(u_1) + \ell_S (w_0) + \ell_S(u_2)$.
\item[(3)]
Let $w_0$ be the element of minimal length in $W_X\,w\,W_Y$.
For each $u_1 \in W_X$ we have $\ell_S(u_1w_0) = \ell_S(u_1) + \ell_S(w_0)$, and for each $u_2 \in W_Y$ we have $\ell_S(w_0u_2) = \ell_S(w_0) + \ell_S(u_2)$.
\end{itemize}
\end{lem}

Let $X, Y \subset V(\Gamma)$ and $w_0 \in W$.
We say that $w_0$ is \emph{$(X,Y)$-minimal} if it is of minimal length in the double-coset $W_X\, w_0\, W_Y$.

The first ingredient in the proof of Theorem \ref{thm1_1} is the following.

\begin{lem}\label{lem2_2}
Let $X, Y \subset V(\Gamma)$ and $w\in W$ such that $w W_Y w^{-1} \subset W_X$.
Then there exist $Y' \subset X$ and $\alpha\in A_X$ such that $\iota(w)\, A_Y\,\iota(w)^{-1} = \alpha A_{Y'} \alpha^{-1}$. In particular, $\iota(w)\, A_Y\,\iota(w)^{-1} \subset A_X$.
\end{lem}

\begin{proof}
Let $w_0$ be the element of minimal length in the double-coset $W_X\,w\,W_Y$.
By Lemma \ref{lem2_1} there exist $u_1 \in W_X$ and $u_2 \in W_Y$ such that $w = u_1 w_0 u_2$ and $\ell_S (w) = \ell_S (u_1) + \ell_S (w_0) + \ell_S (u_2)$.
Since $w W_Y w^{-1} \subset W_X$, $u_1 \in W_X$ and $u_2 \in W_Y$, we have $w_0 W_Y w_0^{-1} \subset W_X$.

Let $y \in Y$, and let $\psi(y) = w_0 s_y w_0^{-1} \in W_X$.
We have that $w_0 s_y = \psi(y)\, w_0$.
Furthermore, by Lemma \ref{lem2_1}\,(3), we have $\ell_S(w_0)+1 = \ell_S(w_0s_y) = \ell_S(\psi(y)\,w_0) = \ell_S(\psi(y)) + \ell_S (w_0)$, and hence $\ell_S(\psi(y))=1$.
So, there exists $f(y) \in X$ such that $w_0 s_y w_0^{-1} = \psi(y) =  s_{f(y)}$.
Note that the above defined map $f : Y \to X$ is injective since conjugation by $w_0$ is an automorphism. 
We set $Y' = f(Y) \subset X$.

Let $y \in Y$.
We have $w_0 s_y = s_{f(y)} w_0$ and  $\ell_S(w_0 s_y) = \ell_S(s_{f(y)}w_0) = \ell_S(w_0)+1$, hence
\[
\iota(w_0)\,\sigma_y = 
\iota(w_0)\,\iota(s_y) =
\iota(w_0s_y) =
\iota (s_{f(y)}w_0) =
\iota (s_{f(y)})\,\iota(w_0) =
\sigma_{f(y)}\,\iota(w_0)\,.
\]
This implies that $\iota(w_0)\, \Sigma_Y\, \iota(w_0)^{-1} = \Sigma_{Y'}$, thus $\iota(w_0)\, A_Y\, \iota(w_0)^{-1} = A_{Y'}$.

We set $\alpha = \iota(u_1) \in A_X$.
Then, since $\iota(u_2) \in A_Y$,
\begin{gather*}
\iota(w)\,A_Y\,\iota(w)^{-1} =
\iota(u_1)\,\iota(w_0)\,\iota(u_2)\,A_Y\,\iota(u_2)^{-1} \iota(w_0)^{-1} \iota(u_1)^{-1} =\\
\iota(u_1)\,\iota(w_0)\,A_Y\,\iota(w_0)^{-1} \iota(u_1)^{-1} =
\iota(u_1)\,A_{Y'}\,\iota(u_1)^{-1} =
\alpha A_{Y'} \alpha^{-1}\,.
\end{gather*}
\end{proof}

We now turn to construct a set-retraction of the inclusion map from $A_X$ into $A$, that is, a map $\pi_X: A \to A_X$ which satisfies $\pi_X(\alpha) = \alpha$ for all $\alpha \in A_X$.
This map will be used to prove Lemma \ref{lem2_4} which is the second and last ingredient in the proof of Theorem \ref{thm1_1}.
Note that the main ideas of the proof of Proposition \ref{prop2_3} come from the proof of Theorem 1.2 of Charney--Paris \cite{ChaPar1}.

Recall that $(\Sigma \sqcup \Sigma^{-1})^*$ denotes the free monoid freely generated by $\Sigma \sqcup \Sigma^{-1}$, that is, the set of words over the alphabet $\Sigma \sqcup \Sigma^{-1}$.
Let $X \subset V(\Gamma)$.
Let $\hat \alpha =\sigma_{z_1}^{\varepsilon_1} \sigma_{z_2}^{\varepsilon_2} \cdots \sigma_{z_p}^{\varepsilon_p} \in (\Sigma \sqcup \Sigma^{-1})^*$.
We set $u_0 = 1 \in W$ and, for $i \in \{1, \dots, p\}$, we set $u_i = s_{z_1} s_{z_2} \cdots s_{z_i} \in W$.
We write each $u_i$ in the form $u_i = v_i w_i$ where $v_i \in W_X$ and $w_i$ is $(X, \emptyset)$-minimal.
Let $i \in \{1, \dots, p\}$.
We set $t_i=w_{i-1} s_{z_i} w_{i-1}^{-1}$ if $\varepsilon_i=1$ and $t_i=w_i s_{z_i} w_i^{-1}$ if $\varepsilon_i=-1$.
If $t_i \not\in S_X$, then we set $\tau_i=1$.
Suppose that $t_i \in S_X$, and let $x_i \in X$ such that $t_i = s_{x_i}$.
Then we set $\tau_i = \sigma_{x_i}^{\varepsilon_i}$.
Finally, we set 
\[
\hat \pi_X (\hat \alpha) = \tau_1 \tau_2 \cdots \tau_p \in (\Sigma_X \sqcup \Sigma_X^{-1})^*\,.
\]
While the definition of $\hat \pi_X$ may seem ad hoc at first, it will become clear in the proof of the following proposition.

\begin{prop}\label{prop2_3}
Let $X \subset V(\Gamma)$.
\begin{itemize}
\item[(1)]
Let $\hat \alpha, \hat \beta \in (\Sigma \sqcup \Sigma^{-1})^*$.
If $\hat \alpha$ and $\hat \beta$ represent the same element of $A$, then $\hat \pi_X (\hat \alpha)$ and $\hat \pi_X (\hat \beta)$ represent the same element of $A_X$.
In other words, the map $\hat \pi_X : (\Sigma \sqcup \Sigma^{-1})^* \to (\Sigma_X \sqcup \Sigma_X^{-1})^*$ induces a set-map $\pi_X : A \to A_X$.
\item[(2)]
We have $\pi_X(\alpha) = \alpha$ for all $\alpha \in A_X$.
\item[(3)]
The restriction of $\pi_X$ to $\CA$ is a homomorphism $\pi_X : \CA \to \CA_X$.
\end{itemize}
\end{prop}

\begin{proof}
The \emph{Salvetti complex} of $\Gamma$ is a CW-complex $\SSal(\Gamma)$ whose $2$-skeleton coincides with the $2$-complex associated with the standard presentation of $A$ (see Godelle--Paris \cite{GodPar1}, Paris \cite{Paris2}, Salvetti \cite{Salve1} or Charney--Davis \cite{ChaDav1}).
In particular, $\SSal(\Gamma)$ has a unique vertex $o_0$, and it has one edge $\bar a_x$ for each $x \in V(\Gamma)$.
We also have an isomorphism $A \to \pi_1(\SSal(\Gamma))$ which sends $\sigma_x$ to the homotopy class of $\bar a_x$ for all $x \in V(\Gamma)$.
Let $p: \Sal (\Gamma) \to \SSal(\Gamma)$ be the regular covering associated with $\theta : A \to W$.
The set of vertices of $\Sal (\Gamma)$ is a set $\{ o(u) \mid u \in W \}$ in one-to-one correspondence with $W$ and the set of edges is a set $\{ a_x(u) \mid x \in V(\Gamma)\,,\ u \in W\}$ in one-to-one correspondence with $V(\Gamma) \times W$.
An edge $a_x(u)$ connects $o(u)$ with $o(us_x)$, and it is assumed to be oriented from $o(u)$ to $o(us_x)$.
We have $p(o(u))=o_0$ for all $u \in W$ and $p(a_x(u))=\bar a_x$ for all $(x,u) \in V(\Gamma) \times W$.
We have an action of $W$ on $\Sal (\Gamma)$ and $\Sal(\Gamma) / W = \SSal (\Gamma)$.
This action is defined on the vertices and edges as follows:
\[
v\,o(u)=o(vu)\,,\ v\,a_x(u) = a_x(vu)\,.
\]

Let $X \subset V(\Gamma)$.
We have an embedding $\bar \nu_X : \SSal(\Gamma_X) \to  \SSal(\Gamma)$ which sends $\bar a_x$ to $\bar a_x$ for all $x \in X$ and which induces the natural embedding of $A_X$ into $A$.
We also have an embedding $\nu_X : \Sal (\Gamma_X) \to \Sal(\Gamma)$ which sends $o(u)$ to $o(u)$ for all $u \in W_X$, which sends $a_x(u)$ to $a_x(u)$ for all $(x,u) \in X \times W_X$, and which induces the natural embedding of $\CA_X$ into $\CA$.
These two embeddings are linked with the following commutative diagram:
\[
\xymatrix{
\Sal(\Gamma_X) \ar[r]^{\nu_X} \ar[d]^{p} & \Sal(\Gamma) \ar[d]^p\\
\SSal(\Gamma_X) \ar[r]^{\bar \nu_X} & \SSal(\Gamma)}
\]

We know by Godelle--Paris \cite[Theorem 2.2]{GodPar1} that the embedding $\nu_X : \Sal(\Gamma_X) \to \Sal(\Gamma)$ admits a retraction $\rho_X : \Sal(\Gamma) \to \Sal(\Gamma_X)$.
This retraction is cellular in the sense that it sends the $k$-skeleton of $\Sal (\Gamma)$ to the $k$-skeleton of $\Sal(\Gamma_X)$ for all $k \ge 0$.
The following explicit description of $\rho_X$ on the $0$ and $1$-skeletons of $\Sal (\Gamma)$ is proved in Charney--Paris \cite[Lemma 2.6]{ChaPar1}. 
Let $u \in W$ and $z \in V(\Gamma)$.
We write $u$ in the form $u=vw$ where $v \in W_X$ and $w$ is $(X,\emptyset)$-minimal.
\begin{itemize}
\item
$\rho_X(o(u)) = o(v)$.
\item
If $w s_z w^{-1} \not\in S_X$, then $\rho_X(a_z(u)) = o(v)$.
\item
Suppose that $w s_z w^{-1} \in S_X$.
Let $x \in X$ such that $w s_z w^{-1} = s_x$.
Then $\rho_X(a_z(u)) = a_x (v)$.
\end{itemize}

In what follows we compose paths from left to right.
Let $\hat \alpha = \sigma_{z_1}^{\varepsilon_1} \sigma_{z_2}^{\varepsilon_2} \cdots \sigma_{z_p}^{\varepsilon_p} \in (\Sigma \sqcup \Sigma^{-1})^*$.
Let 
\[
\bar \gamma (\hat \alpha) = \bar a_{z_1}^{\varepsilon_1} \bar a_{z_2}^{\varepsilon_2} \cdots \bar a_{z_p}^{\varepsilon_p}\,.
\]
We see that, if $\alpha$ is the element of $A$ represented by $\hat \alpha$, then $\alpha$, regarded as an element of $\pi_1(\SSal(\Gamma)) = A$,  is represented by the loop $\bar \gamma (\hat \alpha)$. 
Let $\gamma (\hat \alpha)$ be the lift of $\bar \gamma(\hat \alpha)$ in $\Sal(\Gamma)$ starting at $o(1)$.
We set $u_0 = 1 \in W$ and, for $i \in \{1, \dots, p\}$, we set $u_i = s_{z_1} s_{z_2} \cdots s_{z_i} \in W$.
For $i\in \{1, \dots, p\}$ we set $a_i=a_{z_i}(u_{i-1})$ if $\varepsilon_i=1$, and $a_i = a_{z_i}(u_i)$ if $\varepsilon_i=-1$.
Then 
\[
\gamma (\hat \alpha) = a_1^{\varepsilon_1} a_2^{\varepsilon_2} \cdots a_p^{\varepsilon_p}\,.
\]
Let $\gamma_X (\hat \alpha) = \rho_X(\gamma (\hat \alpha))$.
We write each $u_i$ in the form $u_i = v_i w_i$ where $v_i \in W_X$ and $w_i$ is $(X, \emptyset)$-minimal.
Let $i \in \{1, \dots, p\}$.
We set $t_i=w_{i-1} s_{z_i} w_{i-1}^{-1}$ if $\varepsilon_i=1$, and $t_i=w_i s_{z_i} w_i^{-1}$ if $\varepsilon_i=-1$.
If $t_i \not\in S_X$, then, as shown in Charney--Paris \cite[Lemma 2.6]{ChaPar1}, $v_i=v_{i-1}$.
In that case we denote by $b_i$ the constant path at $o(v_{i-1}) = o(v_i)$.
Suppose that $t_i \in S_X$.
Let $x_i \in X$ such that $t_i = s_{x_i}$.
We set $b_i = a_{x_i}(v_{i-1})$ if $\varepsilon_i=1$, and $b_i = a_{x_i}(v_i)^{-1}$  if $\varepsilon_i=-1$.
It follows from the description of the map $\rho_X$ on the $0$ and $1$-skeletons given above that
\[
\gamma_X (\hat \alpha) = b_1 b_2 \cdots b_p\,.
\]
Let $\bar \gamma_X(\hat \alpha) = p(\gamma_X (\hat \alpha))$.
Let $i \in \{1, \dots, p\}$.
If $t_i \not \in S_X$, then we denote by $\bar b_i$ the constant loop in $\SSal(\Gamma_X)$ based at $o_0$.
Suppose $t_i \in S_X$.
Let $x_i \in X$ such that $t_i = s_{x_i}$ as before.
We set $\bar b_i = \bar a_{x_i}$ if $\varepsilon_i=1$, and $\bar b_i = \bar a_{x_i}^{-1}$ if $\varepsilon_i=-1$.
Then
\[
\bar \gamma_X (\hat \alpha) = \bar b_1 \bar b_2 \cdots \bar b_p\,.
\]
Let $\alpha' \in A_X = \pi_1 (\SSal (\Gamma_X))$ be the element represented by the loop $\bar \gamma_X (\hat \alpha)$.
Then we easily see that $\alpha'$ is exactly the element of $A_X$ represented by the word $\hat \pi_X(\hat \alpha) \in (\Sigma_X \sqcup \Sigma_X^{-1})^*$.

{\it Proof of Part (1).}
Let $\hat \alpha, \hat \beta \in (\Sigma \sqcup \Sigma^{-1})^*$ be two words that represent the same element of $A$. 
Then $\bar \gamma(\hat \alpha)$ and $\bar \gamma( \hat \beta)$ represent the same element of $A = \pi_1 (\SSal (\Gamma))$, hence $\bar \gamma (\hat \alpha)$ and $\bar \gamma (\hat \beta)$ are homotopic loops. 
Since $p : \Sal (\Gamma) \to \SSal(\Gamma)$ is a covering map, $\gamma (\hat \alpha)$ and $\gamma (\hat \beta)$ are homotopic relative to the extremities.
Since $\rho_X$ is continuous, it follows that $\gamma_X(\hat \alpha)$ and $\gamma_X (\hat \beta)$ are also homotopic relative to the extremities. 
Again, the map $p : \Sal(\Gamma_X) \to \SSal (\Gamma_X)$ is continuous, hence $\bar \gamma_X (\hat \alpha)$ and $\bar \gamma_X (\hat \beta)$ are homotopic loops, and therefore they represent the same element of $A_X = \pi_1 (\SSal (\Gamma_X))$.
We conclude that $\hat \pi_X (\hat \alpha)$ and $\hat \pi_X (\hat \beta)$ represent the same element of $A_X$.

{\it Proof of Part (2).}
Let $\alpha \in A_X$.
We choose a word $\hat \alpha = \sigma_{x_1}^{\varepsilon_1} \sigma_{x_2}^{\varepsilon_2} \cdots \sigma_{x_p}^{\varepsilon_p} \in (\Sigma_X \sqcup \Sigma_X^{-1})^*$ which represents $\alpha$.
Following the above definition, we set $u_0 = 1$ and, for $i \in \{1, \dots, p\}$, we set $u_i = s_{x_1} s_{x_2} \cdots s_{x_i}$.
We write each $u_i$ in the form $u_i = v_i w_i$ where $v_i \in W_X$ and $w_i$ is $(X, \emptyset)$-minimal.
Note that $u_i \in W_X$, hence $v_i = u_i$ and $w_i = 1$.
Let $i \in \{1, \dots, p\}$.
We set $t_i=w_{i-1} s_{x_i} w_{i-1}^{-1}$ if $\varepsilon_i=1$, and $t_i=w_i s_{x_i} w_i^{-1}$ if $\varepsilon_i=-1$.
In both cases we have $t_i = s_{x_i}$, and so $\tau_i = \sigma_{x_i}^{\varepsilon_i}$.
So, 
\[
\hat \pi_X (\hat \alpha) = \tau_1 \tau_2 \cdots \tau_p 
= \sigma_{x_1}^{\varepsilon_1} \sigma_{x_2}^{\varepsilon_2} \cdots \sigma_{x_p}^{\varepsilon_p} = \hat \alpha\,,
\]
hence $\pi_X(\alpha) = \alpha$.

{\it Proof of Part (3).}
Observe that the restriction of $\pi_X$ to $\CA$ coincides with the homomorphism $\rho_{X,*} : \CA = \pi_1 (\Sal(\Gamma)) \to \pi_1 (\Sal(\Gamma_X))= \CA_X$ induced by the map $\rho_X : \Sal(\Gamma) \to \Sal(\Gamma_X)$.
To see this, note that $\rho_X$ does to edge paths in $\Sal(\Gamma)$ what $\hat \pi_X$ does to elements in $(\Sigma \sqcup \Sigma^{-1})^*$ (where the $\varepsilon$ appearing in the definition of $\hat \pi_X$ reflect the orientation of the edges in $\Sal(\Gamma)$).
Hence, the restriction of $\pi_X$ to $\CA$ is a homomorphism $\pi_X : \CA \to \CA_X$.
\end{proof}

Now, thanks to Proposition \ref{prop2_3} we can prove the second ingredient of the proof of Theorem \ref{thm1_1}.

\begin{lem}\label{lem2_4}
Let $X \subset V(\Gamma)$, $\alpha \in A_X$ and $\beta \in \CA$.
If $\beta \alpha \beta^{-1} \in A_X$, then $\beta \alpha \beta^{-1} = \pi_X(\beta)\, \alpha\, \pi_X(\beta)^{-1}$.
\end{lem}

\begin{proof}
We assume that $\beta \alpha \beta^{-1} \in A_X$.
We choose a word $\sigma_{z_1}^{\varepsilon_1} \sigma_{z_2}^{\varepsilon_2} \cdots \sigma_{z_p}^{\varepsilon_p} \in (\Sigma \sqcup \Sigma^{-1})^*$ which represents $\beta$ and a word $\sigma_{x_1}^{\mu_1} \sigma_{x_2}^{\mu_2} \cdots \sigma_{x_q}^{\mu_q} \in (\Sigma_X \sqcup \Sigma_X^{-1})^*$ which represents $\alpha$.
We start with the definition of $\pi_X(\beta \alpha \beta^{-1})$ which uses the representative word $\sigma_{z_1}^{\varepsilon_1} \cdots \sigma_{z_p}^{\varepsilon_p} \sigma_{x_1}^{\mu_1} \cdots \sigma_{x_q}^{\mu_q} \sigma_{z_p}^{-\varepsilon_p} \cdots \sigma_{z_1}^{-\varepsilon_1}$.
We set $u_{0,1}=1$ and, for $i \in \{1, \dots, p\}$, we set $u_{i,1}=s_{z_1} s_{z_2} \cdots s_{z_i}$.
We write each $u_{i,1}$ in the form $u_{i,1} = v_{i,1} w_{i,1}$ where $v_{i,1} \in W_X$ and $w_{i,1}$ is $(X, \emptyset)$-minimal.
Let $i\in\{1,\dots,p\}$.
We set $t_{i,1}=w_{i-1,1} s_{z_i} w_{i-1,1}^{-1}$ if $\varepsilon_i=1$, and $t_{i,1}=w_{i,1} s_{z_i} w_{i,1}^{-1}$ if $\varepsilon_i=-1$.
We set $\tau_{i,1} = 1$ if $t_{i,1} \not\in S_X$, and $\tau_{i,1}=\sigma_{x_{i,1}}^{\varepsilon_i}$ if $t_{i,1}\in S_X$, where $x_{i,1}$ is the element of $X$ such that $t_{i,1} = s_{x_{i,1}}$.
We set $u_{0,2}=\theta(\beta)$ and, for $i \in \{1, \dots, q\}$, we set $u_{i,2}=\theta(\beta)\,s_{x_1} s_{x_2} \cdots s_{x_i}$.
We write each $u_{i,2}$ in the form $u_{i,2} = v_{i,2} w_{i,2}$, where $v_{i,2} \in W_X$ and $w_{i,2}$ is $(X, \emptyset)$-minimal.
Let $i \in \{1,\dots,q\}$.
We set $t_{i,2}=w_{i-1,2} s_{x_i} w_{i-1,2}^{-1}$ if $\mu_i=1$, and $t_{i,2}=w_{i,2} s_{x_i} w_{i,2}^{-1}$ if $\mu_i=-1$.
We set $\tau_{i,2} = 1$ if $t_{i,2} \not\in S_X$, and $\tau_{i,2}=\sigma_{x_{i,2}}^{\mu_i}$ if  $t_{i,2}\in S_X$, where $x_{i,2}$ is the element of $X$ such that $t_{i,2} = s_{x_{i,2}}$.
We set $u_{p+1,3}=\theta (\beta)\, \theta(\alpha)$ and, for $i \in \{1, \dots, p\}$, we set $u_{i,3}= \theta(\beta)\, \theta(\alpha)\, s_{z_p} s_{z_{p-1}} \cdots s_{z_i}$.
We write each $u_{i,3}$ in the form $u_{i,3} = v_{i,3} w_{i,3}$, where $v_{i,3} \in W_X$ and $w_{i,3}$ is $(X, \emptyset)$-minimal.
Let $i \in \{1,\dots, p\}$.
We set $t_{i,3}=w_{i+1,3} s_{z_i} w_{i+1,3}^{-1}$ if $\varepsilon_i=-1$, and $t_{i,3}=w_{i,3} s_{z_i} w_{i,3}^{-1}$ if $\varepsilon_i=1$.
We set $\tau_{i,3} = 1$ if $t_{i,3} \not\in S_X$, and $\tau_{i,3}=\sigma_{x_{i,3}}^{-\varepsilon_i}$ if $t_{i,3}\in S_X$, where $x_{i,3}$ is the element of $X$ such that $t_{i,3} = s_{x_{i,3}}$.
Then, by definition, 
\[
\pi_X(\beta \alpha \beta^{-1}) = \tau_{1,1} \tau_{2,1} \cdots \tau_{p,1} \tau_{1,2} \tau_{2,2} \cdots \tau_{q,2} \tau_{p,3} \cdots \tau_{2,3} \tau_{1,3}\,.
\]
We also have $\pi_X(\beta \alpha \beta^{-1}) = \beta \alpha \beta^{-1}$, since $\beta \alpha \beta^{-1} \in A_X$.

We have $\tau_{1,1} \tau_{2,1} \cdots \tau_{p,1} = \pi_X(\beta)$ by definition.
Let $i \in \{0,1, \dots, q\}$.
We have $\theta(\beta)=1$ since $\beta \in \CA$, hence $u_{i,2}=s_{x_1} s_{x_2} \cdots s_{x_i} \in W_X$.
It follows that $v_{i,2} = u_{i,2}$ and $w_{i,2}=1$.
Let $i \in \{1, \dots, q\}$.
Then $t_{i,2}=s_{x_i} \in S_X$ and $\tau_{i,2} = \sigma_{x_i}^{\mu_i}$.
So,
\[
\tau_{1,2} \tau_{2,2} \cdots \tau_{q,2} =
\sigma_{x_1}^{\mu_1} \sigma_{x_2}^{\mu_2} \cdots \sigma_{x_q}^{\mu_q} =
\alpha\,.
\]
Let $i \in \{0,1, \dots ,p\}$.
We have $1 = \theta (\beta) = s_{z_1} \cdots s_{z_i} s_{z_{i+1}} \cdots s_{z_p}$, hence $s_{z_p} \cdots s_{z_{i+1}} = s_{z_1} \cdots s_{z_i} = u_{i,1}$, and therefore 
\[
u_{i,3} = 
\theta(\beta)\,\theta(\alpha)\, s_{z_p} \cdots s_{z_i} =
\theta (\alpha)\, s_{z_1} \cdots s_{z_{i-1}} =
\theta(\alpha)\, u_{i-1,1} =
\theta(\alpha)\, v_{i-1,1} w_{i-1,1}\,.
\]
Since $\theta(\alpha) \in W_X$, it follows that $v_{i,3} = \theta(\alpha)\, v_{i-1,1}$ and $w_{i,3} = w_{i-1,1}$.
Let $i \in \{1, \dots, p\}$.
If $\varepsilon_i=1$, then  
\[
t_{i,3} = 
w_{i,3} s_{z_i} w_{i,3}^{-1} =
w_{i-1,1} s_{z_i} w_{i-1,1}^{-1} =
t_{i,1}\,.
\]
Similarly, if $\varepsilon_i=-1$, then  
\[
t_{i,3} = 
w_{i+1,3} s_{z_i} w_{i+1,3}^{-1} =
w_{i,1} s_{z_i} w_{i,1}^{-1} =
t_{i,1}\,.
\]
In both cases it follows that $\tau_{i,3} = \tau_{i,1}^{-1}$.
So,
\[
\tau_{p,3} \cdots \tau_{2,3} \tau_{1,3} =
\tau_{p,1}^{-1} \cdots \tau_{2,1}^{-1} \tau_{1,1}^{-1} =
\pi_X(\beta)^{-1}\,.
\]
Finally,
\[
\beta \alpha \beta ^{-1} =
\pi_X(\beta \alpha \beta ^{-1}) =
\pi_X(\beta)\, \alpha\, \pi_X(\beta)^{-1}\,.
\]
\end{proof}

\begin{proof}[Proof of Theorem \ref{thm1_1}]
Let $X,Y \subset V(\Gamma)$ and $\alpha \in A$ such that $\alpha A_Y \alpha^{-1} \subset A_X$.
Let $w = \theta(\alpha)$.
We have $wW_Yw^{-1} \subset W_X$, hence, by Lemma \ref{lem2_2}, there exist $Y'\subset X$ and $\beta_2 \in A_X$ such that $\iota(w)\,A_Y\,\iota(w)^{-1} = \beta_2 A_{Y'} \beta_2^{-1}$.
Let $\beta_1=\alpha\,\iota(w)^{-1}$.
Then 
\[
\alpha A_Y \alpha^{-1} =
\alpha \, \iota(w)^{-1}\, \iota(w)\, A_Y\, \iota(w)^{-1}\, \iota(w)\, \alpha^{-1} =
\beta_1 \beta_2 A_{Y'} \beta_2^{-1} \beta_1^{-1}\,.
\]
We have $\beta_1 \in \CA$, since $\theta(\beta_1)=ww^{-1}=1$, $\beta_2 A_{Y'} \beta_2^{-1} \subset A_X$ and $\beta_1 (\beta_2 A_{Y'} \beta_2^{-1}) \beta_1^{-1} \subset A_X$, hence, by Lemma \ref{lem2_4},
\[
\alpha A_Y \alpha^{-1} =
\beta_1 (\beta_2 A_{Y'} \beta_2^{-1}) \beta_1^{-1} =
\pi_X(\beta_1)\, (\beta_2 A_{Y'} \beta_2^{-1})\, \pi_X(\beta_1)^{-1}\,.
\]
So, if $\gamma =  \pi_X(\beta_1)\,\beta_2$, then $\gamma \in A_X$ and $\alpha A_Y \alpha^{-1} = \gamma A_{Y'} \gamma^{-1}$.
\end{proof}



\end{document}